\documentclass[11pt]{article}
\usepackage[english]{babel}
\usepackage{amssymb,amsmath,amsthm}
\usepackage[dvipsnames]{pstricks}
\newtheorem{theorem}{Theorem}[section]
\newtheorem{lemma}[theorem]{Lemma}
\newtheorem{proposition}[theorem]{Proposition}
\newtheorem{corollary}[theorem]{Corollary}
\newtheorem{question}[theorem]{Question}

{\theoremstyle{definition}\newtheorem{definition}[theorem]{Definition}}
{\theoremstyle{definition}\newtheorem{remark}[theorem]{Remark}}

\numberwithin{equation}{section}

\def\C{{\mathbb C}}
\def\N{{\mathbb N}}
\def\Z{{\mathbb Z}}
\def\R{{\mathbb R}}
\def\D{{\mathbb D}}
\def\T{{\mathbb T}}
\def\DD{\overline{\mathbb D}}
\def\hh{{\mathcal H}}
\def\pp{{\mathbb P}}

\def\aa{{\mathbb A}}
\def\epsilon{\varepsilon}
\def\phi{\varphi}
\def\leq{\leqslant}
\def\geq{\geqslant}
\def\ker{{\tt ker}\,}
\def\spann{\hbox{\tt span}\,}
\def\im{\hbox{\tt Im}\,}
\def\det{{\tt det}\,}

\def\ilim{\mathop{\hbox{$\underline{\hbox{\rm lim}}$}}\limits}
\def\slim{\mathop{\hbox{$\overline{\hbox{\rm lim}}$}}\limits}
\def\ssub#1#2{#1_{{}_{{\scriptstyle #2}}}}

\title{Chaotic Banach algebras}

\author{Stanislav Shkarin}

\date{}

\begin{document}

\maketitle

\begin{abstract}
We construct an infinite dimensional non-unital Banach algebra $A$
and $a\in A$ such that the sets $\{za^n:z\in\C,\ n\in\N\}$ and
$\{({\bf 1}+a)^na:n\in\N\}$ are both dense in $A$, where $\bf 1$ is
the unity in the unitalization $A^{\#}=A\oplus \spann\{{\bf 1}\}$ of
$A$. As a byproduct, we get a hypercyclic operator $T$ on a Banach
space such that $T\oplus T$ is non-cyclic and $\sigma(T)=\{1\}$.
\end{abstract}


\section{Introduction \label{s1}}\rm

All vector spaces in this article are over the field $\C$ of complex
numbers. As usual, $\R$ is the field of real numbers,
$\T=\{x\in\C:|z|=1\}$, $\D=\{z\in\C:|z|<1\}$, $\DD=\{z\in\C:|z|\leq
1\}$, $\R_+=[0,\infty)$, $\N$ is the set of positive integers and
$\Z_+=\N\cup\{0\}$. If $X$ and $Y$ are topological vector spaces,
$L(X,Y)$ stands for the space of continuous linear operators from
$X$ to $Y$. We write $L(X)$ instead of $L(X,X)$ and $X^*$ instead of
$L(X,\C)$. For $T\in L(X,Y)$, the dual operator $T^*\in L(Y^*,X^*)$
is defined as usual: $T^*f=f\circ T$. Recall that $T\in L(X)$ is
called {\it hypercyclic} (respectively, {\it supercyclic}) if there
is $x\in X$ such that the {\it orbit} $O(T,x)=\{T^nx:n\in\Z_+\}$
(respectively, the {\it projective orbit} $\{zT^nx:z\in\C,\
n\in\Z_+\}$) is dense in $X$. Such an $x$ is called a {\it
hypercyclic vector} (respectively, a {\it supercyclic vector}) for
$T$. We refer to \cite{bama-book} and references therein for
additional information on hypercyclicity and supercyclicity. Recall
that a function $\pi:A\to\R_+$ defined on a complex algebra $A$ is
called {\it submultiplicative} if $\pi(ab)\leq \pi(a)\pi(b)$ for any
$a,b\in A$. A {\it Banach algebra} is a complex (maybe non-unital)
algebra $A$ with a complete submultiplicative norm (if $A$ is
unital, it is usually also assumed that $\|{\bf 1}\|=1$, where $\bf
1$ is the unity in $A$). We say that $A$ is {\it non-trivial} if
$A\neq\{0\}$.

\begin{definition}\label{sba} Let $A$ be a Banach algebra.
We say that $A$ is {\it supercyclic} if there is $a\in A$ for which
$\{za^n:z\in\C,\ n\in\N\}$ is dense in $A$. Such an $a$ is called a
{\it supercyclic element} of $A$. We say that $A$ is {\it almost
hypercyclic} if there is $a\in A$ for which $\{({\bf
1}+a)^na:n\in\N\}$ is dense in $A$. Such an $a$ is called an {\it
almost hypercyclic element} of $A$. Finally, we say that a Banach
algebra $A$ is {\it chaotic} if there is $a\in A$ which is a
supercyclic and an almost hypercyclic element of $A$. In other
words, both $\{za^n:z\in\C,\ n\in\N\}$ and $\{({\bf
1}+a)^na:n\in\N\}$ are dense in $A$. Such an $a$ is called a {\it
chaotic element} of $A$.
\end{definition}

In the above definition $\bf 1$ is the unit element in the
unitalization $A^{\#}=A\oplus \spann\{{\bf 1}\}$ of $A$. Note that
$a$ is a supercyclic element of $A$ if and only if $a$ is a
supercyclic vector for the multiplication operator
\begin{equation}\label{ma}
\text{$M_a\in L(A)$, \ \ $M_ab=ab$}
\end{equation}
and $a$ is an almost hypercyclic element of $A$ if and only if $a$
is a hypercyclic vector for $I+M_a$. There is no point to consider
'hypercyclic Banach algebras' in the obvious sense. Indeed, in
\cite{shka} it is observed that a multiplication operator on a
commutative Banach algebra is never hypercyclic. Obviously,
supercyclic as well as almost hypercyclic Banach algebras are
commutative and separable.

\begin{theorem}\label{main} There exists a chaotic infinite dimensional
Banach algebra $A$.
\end{theorem}

In order to emphasize the value of Theorem~\ref{main}, we would like
to mention few related facts. A Banach algebra is called {\it
radical} if it coincides with its Jackobson radical \cite{dales}. If
$A$ is a Banach algebra and $X$ is a Banach $A$-bimodule
\cite{dales}, then $D\in L(A,X)$ is called a {\it derivation} if
$D(ab)=(Da)b+a(Db)$ for each $a,b\in A$. A Banach algebra $A$ is
called {\it weakly amenable} if every derivation $D:A\to A^*$ (with
the natural bimodule structure on $A^*$) has the shape $Da=ax-xa$
for some $x\in A^*$. It is well-known \cite{dales} that a
commutative Banach algebra $A$ is weakly amenable if and only if
there is no non-zero derivations $D:A\to X$ taking values in a
commutative Banach $A$-bimodule $X$.

\begin{theorem}\label{prop}
Let $A$ be a supercyclic Banach algebra of dimension $>1$. Then $A$
is infinite dimensional, radical and weakly amenable.
\end{theorem}

According to Theorem~\ref{prop}, Theorem~\ref{main} provides an
infinite dimensional radical weakly amenable Banach algebra. We
would like to mention the work \cite{read1} by Loy, Read, Runde, and
Willis, who constructed a non-unital Banach algebra, generated by
one element $x$ and which has a bounded approximate identity of the
shape $x^{n_k}/\|x^{n_k}\|$, where $\{n_k\}_{k\in\N}$ is a strictly
increasing sequence of positive integers. Such an algebra is
automatically radical and weakly amenable. Theorem~\ref{prop} shows
that the same properties are forced by supercyclicity. It is also
worth mentioning that Read \cite{read2} constructed a commutative
amenable radical Banach algebra, but this algebra is not generated
by one element.

\begin{proposition}\label{qui} Let $A$ be a non-trivial commutative Banach
algebra and $M=cI+M_a\in L(A)$, where $a\in A$ and $c\in\C$. Then
$M\oplus M$ is non-cyclic.
\end{proposition}

\begin{proof}Let $(x,y)\in A^2$. If $M_x=M_y=0$, then
$(M\oplus M)^n(x,y)=c^n(x,y)$ for every $n\in\Z_+$ and therefore
$(x,y)$ is not a cyclic vector for $M_a\oplus M_a$. Otherwise, the
operator $T\in L(A^2,A)$, $T(u,v)=yu-xv$ is non-zero. Moreover,
$T((M\oplus M)^n(x,y))=T((c{\bf 1}+a)^nx,(c{\bf 1}+a)^ny)=y(c{\bf
1}+a)^nx-x(c{\bf 1}+a)^ny=0$ since $A$ is commutative. Thus
$(M\oplus M)^n(x,y)\in\ker T$ for each $n\in\Z_+$. Since $\ker T$ is
a proper closed linear subspace of $A^2$, $(x,y)$ again is not a
cyclic vector for $M\oplus M$.
\end{proof}

By Proposition~\ref{qui}, Theorem~\ref{main} provides hypercyclic
operators $T$ with non-cyclic $T\oplus T$. The existence of such
operators used to be an open problem until De~La~Rosa and Read
\cite{read00} (see also \cite{bama} and \cite{bama-book})
constructed such operators. One can observe that the spectra of the
operators in \cite{read00,bama} contain a disk centered at 0 of
radius $>1$. On the other hand \cite{bama-book}, any separable
infinite dimensional complex Banach space  supports hypercyclic
operators with the spectrum being the singleton $\{1\}$. It remained
unclear whether a hypercyclic operator $T$ with non-cyclic $T\oplus
T$ can have small spectrum. Theorem~\ref{main} provides such an
operator. Indeed, by Theorem~\ref{main}, there are an infinite
dimensional Banach algebra $A$ and $a\in A$ such that $T=I+M_a$ is
hypercyclic. By Theorem~\ref{prop}, $A$ is radical and therefore
$M_a$ is quasinilpotent. Hence the spectrum $\sigma(T)$ of $T$ is
$\{1\}$. Thus we arrive to the following corollary.

\begin{corollary}\label{smallsp} There exists a hypercyclic continuous
linear operator $T$ on an infinite dimensional Banach space such
that $T\oplus T$ is non-cyclic and $\sigma(T)=\{1\}$.
\end{corollary}

It seems to be of independent interest that supercyclic operators
$T$ with non-cyclic $T\oplus T$ can be found among multiplication
operators on commutative Banach algebras, while hypercyclic
operators $T$ with non-cyclic $T\oplus T$ can be of the shape
identity plus a multiplication operator.

\section{Proof of Theorem~\ref{prop}}

Since a Banach space of finite dimension $>1$ supports no
supercyclic operators (see \cite{ww}), a supercyclic Banach algebra
of dimension $>1$ must be infinite dimensional. According to
\cite[Proposition~3.4]{shka}, an infinite dimensional commutative
Banach algebra $B$ is radical if there is $b\in B$ for which the
multiplication operator $M_b$ is supercyclic. Since a supercyclic
Banach algebra of dimension $>1$ is infinite dimensional,
commutative and has a supercyclic multiplication operator, $A$ is
radical.

It remains to show that that $A$ is weakly amenable. Assume the
contrary. Then there is a commutative Banach $A$-bimodule $X$ and a
non-zero derivation $D\in L(A,X)$. Since $A$ is supercyclic, there
is $a\in A$ such that $\{za^n:z\in\C,\ n\in\N\}$ is dense in $A$.
Since $\dim A>1$, $\Omega_m=\{za^n:z\in\C,\ n\geq m\}$ is dense in
$A$ for each $m\in\N$. Consider the operator $M\in L(A,X)$,
$Mb=bDa$. Since $X$ is commutative and $D$ is a derivation, we have
$D(a^n)=na^{n-1}Da$ for $n\geq 2$. If $M=0$, then
$D(a^n)=na^{n-1}Da=nM(a^{n-1})=0$ for $n\geq 2$. Hence $D$ vanishes
on the dense set $\Omega_2$. Since $D$ is continuous, $D=0$, which
is a contradiction. Hence $M\neq 0$ and therefore $M^*\neq 0$. Thus
there is $f\in X^*$ such that $g=M^*f^*$ is a non-zero element of
$A^*$. Then for each $n\in\N$, we have
$g(a^n)=M^*f(a^n)=f(a^nDa)=\frac{f(D(a^{n+1}))}{n+1}$. Hence
$$
\text{$\textstyle|g(a^n)|=\frac{|f(D(a^{n+1}))|}{n+1}\leq
\frac{C\|a^n\|}{n+1}$,\quad where $C=\|D\|\,\|f\|\,\|a\|$.}
$$
Now let $m\in\N$ be such that $\frac{C}{m+1}<\frac{\|g\|}{2}$ and
$W=\bigl\{u\in A:|g(u)|>\frac{\|g\|\|u\|}{2}\bigr\}$. Clearly $W$ is
non-empty and open. By the last display, $\Omega_m\cap
W=\varnothing$, which contradicts the density of $\Omega_m$ in $A$.
This contradiction completes the proof of Theorem~\ref{prop}.

\section{Proof of Theorem~\ref{main}}

From now on, $\pp$ is the algebra $\C[z]$ of polynomials with
complex coefficients in one variable $z$. Clearly,
$\pp_0=\{p\in\pp:p(0)=0\}$ is an ideal in $\pp$ of codimension $1$.
There is a sequence $\{p_n\}_{n\in\N}$ in $\pp_0$ such that
\begin{equation}
\text{$\{p_n:n\in\N\}$ is dense in $\pp_0$ with respect to any
seminorm on $\pp_0$.}\label{ppp1}
\end{equation}
Indeed, (\ref{ppp1}) is satisfied if, for instance, $\{p_n:n\in\N\}$
is the set of all polynomials in $\pp_0$ with coefficients from a
fixed dense countable subset of $\C$, containing $0$.

\begin{lemma}\label{z2}Let $\pi$ be a non-zero submultiplicative
seminorm on $\pp_0$ and $\{p_k\}_{k\in\N}$ is a sequence in $\pp_0$
satisfying $(\ref{ppp1})$. Assume also that there exist sequences
$\{n_k\}_{k\in\N}$ and $\{m_k\}_{k\in\N}$ of positive integers and a
sequence $\{c_k\}_{k\in\N}$ of complex numbers such that
$\pi(c_kz^{n_k}-p_k)\to 0$ and $\pi(z(1+z)^{m_k}-p_k)\to 0$. Then
$\pi$ is a norm and the completion $A$ of $(\pp_0,\pi)$ is an
infinite dimensional chaotic Banach algebra with $z$ as a chaotic
element.
\end{lemma}

\begin{proof} Let $I=\{q\in\pp_0:\pi(q)=0\}$. Since $\pi$ is
submultiplicative, $I$ is an ideal in $\pp_0$ and therefore in
$\pp$. Since $\pi$ is non-zero, $I\neq\pp_0$. Thus $\pp_0/I$ with
the norm $\|q+I\|=\pi(q)$ is a non-trivial complex algebra with a
submultiplicative norm. Since $\pi(z(1+z)^{m_k}-p_k)\to 0$,
(\ref{ppp1}) implies that the operator $p+I\mapsto (1+z)p+I$ on
$\pp_0/I$ is hypercyclic with the hypercyclic vector $z+I$. Since
there is no hypercyclic operator on a non-trivial finite dimensional
normed space \cite{ww}, $\pp_0/I$ is infinite dimensional and
therefore $I$ has infinite codimension in $\pp$. Since the only
ideal in $\pp$ of infinite codimension is $\{0\}$, $I=\{0\}$ and
therefore $\pi$ is a norm.

Thus the completion $A$ of $(\pp_0,\pi)$ is an infinite dimensional
Banach algebra. Conditions $\pi(c_kz^{n_k}-p_k)\to 0$ and
$\pi(z(1+z)^{m_k}-p_k)\to 0$ together with (\ref{ppp1}) imply that
$A$ is chaotic with $z$ as a chaotic element.
\end{proof}

It remains to construct a seminorm on $\pp_0$, which will allow us
to apply Lemma~\ref{z2}.

\subsection{Ideals in $\aa^{[k]}$ and submultiplicative norms on $\pp$}

For $k\in\N$, we consider the commutative Banach algebra $\aa^{[k]}$
of the power series
$$
a=\sum_{n\in\Z_+^k} a_nu_1^{n_1}\dots u_k^{n_k},\ \ \text{with}\ \
\ssub{\|a\|}{[k]}=\sum_{n\in\Z_+^k}|a_n|<\infty
$$
with the natural multiplication. We will treat the elements of
$\aa^{[k]}$ both as power series and as continuous functions
$u\mapsto a(u_1,\dots,u_k)$ on $\DD^k$, holomorphic on $\D^k$. Note
that as a Banach space $\aa^{[k]}$ is $\ell_1(\Z_+^k)$. In
particular, the underlying Banach space of $\aa^{[k]}$ can be
treated as the dual space of $c_0(\Z_+^k)$, which allows us to speak
about the $*$-weak topology on $\aa^{[k]}$.

For a non-empty open subset $U$ of $\C$ we also consider the complex
algebra $\hh_U$ of holomorphic functions $f:U\to \C$ endowed with
the Fr\'echet space topology of uniform convergence on compact
subsets of $U$. For $\gamma>0$, we write $\hh_\gamma$ instead of
$\hh_{\gamma\D}$.

If $\xi=(\xi_1,\dots,\xi_k)\in\pp_0^k$ and $a\in\aa^{[k]}$, we can
consider $a(\xi_1,\dots,\xi_k)$ as a power series
\begin{equation}\label{tra}
a(\xi_1,\dots,\xi_k)(z)=a(\xi_1(z),\dots,\xi_k(z))=\sum\limits_{m=1}^\infty
\alpha_m(a,\xi)z^m,
\end{equation}
which converges uniformly on the compact subsets of the disk
$\gamma(\xi)\D$, where
$$
\gamma(\xi)=\sup\{c>0:\xi_j(c\D)\subseteq\D\ \ \text{for}\ \ 1\leq
j\leq k\}>0.
$$
By the Hadamard formula,
$\slim\limits_{m\to\infty}|\alpha_m(a,\xi)|^{1/m}\leq
\frac1{\gamma(\xi)}$ for each $a\in\aa^{[k]}$. By the uniform
boundedness principle,
$\slim\limits_{m\to\infty}\|\alpha_m(\cdot,\xi)\|^{1/m}\leq
\frac1{\gamma(\xi)}$, where the norm is taken in
$\bigl(\aa^{[k]}\bigr)^*$. Hence the map
$$
\Phi_\xi:\aa^{[k]}\to \hh_{\gamma(\xi)},\quad
\Phi_\xi(a)=a(\xi_1,\dots,\xi_k)
$$
is a continuous algebra homomorphism from the Banach algebra
$\aa^{[k]}$ to the Fr\'echet algebra $\hh_{\gamma(\xi)}$ of
holomorphic complex valued functions on the disk $\gamma(\xi)\D$.

\begin{remark}\label{rema} Note that if $U$ is a connected non-empty open
subset of $\C$ and all zeros of a polynomial $p\in\pp$ of degree
$n\in\N$ are in $U$, then the ideal $J_p$, generated by $p$ in the
algebra $\hh_U$ is closed and has codimension $n$. It consists of
all $f\in\hh_U$ such that every zero of $p$ of order $k\in\N$ is
also a zero of $f$ of order $\geq k$. We write $p\big|f$ to denote
the inclusion $f\in J_p$. Note that
$\hh_U=J_p\oplus\spann\{1,z,\dots,z^{n-1}\}$.
\end{remark}

We use the following notation. If $\xi\in\pp_0^k$ and $q\in\pp$ has
all its zeros in the disk $\gamma(\xi)\D$, then
\begin{equation}\label{ietq}
I_{\xi,q}=\{a\in\aa^{[k]}:q\big|\Phi_\xi(a)\}
\end{equation}
with $\Phi_\xi(a)$ considered as an element of $\hh_{\gamma(\xi)}$.
In the case $q=z^n$ with $n\in\N$, we have
\begin{equation}\label{ietn}
I_{\xi,z^n}=\{a\in\aa^{[k]}:\alpha_j(a,\xi)=0\ \ \text{for}\ \ 0\leq
j<n\},
\end{equation}
where $\alpha_j(a,\xi)$ are defined in (\ref{tra}). Finally,
\begin{equation}\label{iu}
I_{\xi}=\ker\Phi_\xi=\bigcap_{n=1}^\infty I_{\xi,z^n}.
\end{equation}

The proof of the following lemma is lengthy and technical. We
postpone it until the next section.

\begin{lemma}\label{wstarH} Let $\xi=(\xi_1,\dots,\xi_k)\in\pp_0^k$
be such that $\xi_1=z$. Then $I_\xi$ is a closed ideal in
$\aa^{[k]}$ and for each $q\in\pp$, whose zeros are in the disk
$\gamma(\xi)\D$, $I_{\xi,q}$ is closed ideal in $\aa^{[k]}$ of
codimension $\deg q$. Moreover, $I_\xi\subset I_{\xi,q}$ and
\begin{equation}\label{limi1}
\ssub{\|a+I_{\xi,z^n}\|}{A^{[k]}/I_{\xi,z^n}}\to
\ssub{\|a+I_{\xi}\|}{A^{[k]}/I_{\xi}}\ \ \ \text{as $n\to\infty$ for
each $a\in\aa^{[k]}$}.
\end{equation}
Furthermore, if $q_n\in\pp$ for $n\in\N\cup\{\infty\}$ are
polynomials of degree $m\in\N$, whose zeros are in $\gamma(\xi)\D$
and the sequence $\{q_n\}_{n\in\N}$ converges to $q_\infty$ as $n\to
\infty$ $($in the usual sense in the finite dimensional space of
polynomials of degree $\leq m)$, then
\begin{equation}\label{limi2}
\ssub{\|a+I_{\xi,q_n}\|}{A^{[k]}/I_{\xi,q_n}}\to
\ssub{\|a+I_{\xi,q_\infty}\|}{A^{[k]}/I_{\xi,q_\infty}}\ \ \
\text{as $n\to\infty$ for each $a\in\aa^{[k]}$}.
\end{equation}
\end{lemma}

If $\xi\in\pp_0^k$ and $\xi_1=z$, then
$\pp\subseteq\Phi_\xi(\aa^{[k]})$. Indeed, $\Phi_\xi(a)=p$ if
$p\in\pp$ and $a(u_1,\dots,u_k)=p(u_1)$. Hence we can use the above
ideals to define seminorms on $\pp$. Since $I_\xi=\ker\Phi_\xi$ and
$\Phi_\xi(\aa^{[k]})\supseteq \pp$, we can define
\begin{equation}\label{piu}
\pi_\xi:\pp\to \R_+, \quad
\pi_\xi(p)=\ssub{\|\Phi_\xi^{-1}(p)\|}{\aa^{[k]}/I_{\xi}}=
\inf\{\ssub{\|a\|}{[k]}:a\in\aa^{[k]},\ \Phi_\xi(a)=p\}.
\end{equation}
By Lemma~\ref{wstarH}, $I_{\xi}$ is a closed ideal in $\aa^{[k]}$
and therefore $\pi_\xi$ is a submultiplicative norm on $\pp$.

If additionally $q\in\pp$ has all its zeros in the disk
$\gamma(\xi)\D$, then using the closeness of the ideal $I_{\xi,q}$
in $\aa^{[k]}$ and the inclusion $I_\xi\subset I_{\xi,q}$,  we can
define
\begin{equation}\label{piuq}
\pi_{\xi,q}:\pp\to \R_+,\quad
\pi_{\xi,q}(p)=\ssub{\|\Phi_\xi^{-1}(p)+I_{\xi,q}\|}{\aa^{[k]}/I_{\xi,q}}=
\inf\{\ssub{\|a\|}{[k]}:a\in\aa^{[k]},\ q\big|(p-\Phi_\xi(a))\}.
\end{equation}
The function $\pi_{\xi,q}$ is a submultiplicative seminorm on $\pp$.

\begin{lemma}\label{kaka}Let $k\in\N$,
$\xi'=(\xi_1,\dots,\xi_{k+1})\in\pp_0^{k+1}$ with $\xi_1=z$ and
$\xi=(\xi_1,\dots,\xi_k)$. Then $\pi_{\xi'}(p)\leq \pi_\xi(p)$ for
all $p\in\pp$. Moreover, if $U$ is a connected open subset of
$\gamma(\xi)\D$, $0\in U$, $\xi_{k+1}(U)\subseteq\D$ and
$q\in\pp\setminus\{0\}$ is a divisor of $\xi_{k+1}$ and has all its
zeros in $U$, then $\pi_{\xi,q}(p)\leq \pi_{\xi'}(p)$ for every
$p\in\pp$.
\end{lemma}

\begin{proof} For any $p\in\pp$ and $a\in\aa^{[k]}$ satisfying
$\Phi_\xi(a)=p$, we have $\Phi_{\xi'}(b)=p$ and
$\ssub{\|a\|}{[k]}=\ssub{\|b\|}{[k+1]}$ with
$b(u_1,\dots,u_{k+1})=a(u_1,\dots,u_k)$. By (\ref{piu}),
$\pi_{\xi'}(p)\leq \pi_\xi(p)$ for each $p\in\pp$. Now assume that
$U$ is a connected open subset of $\gamma(\xi)\D$, $0\in U$,
$\xi_{k+1}(U)\subseteq\D$ and $q\in\pp\setminus\{0\}$ is a divisor
of $\xi_{k+1}$ and has all its zeros in $U$. Let $p\in\pp$ and
$a\in\aa^{[k+1]}$ be such that $\Phi_{\xi'}(a)=p$. By definition of
$\aa^{[k+1]}$,
\begin{equation}\label{ab}
a=b_0+\sum\limits_{n=1}^\infty b_n u_{k+1}^n,\ \ \text{where
$b_j\in\aa^{[k]}$ and}\ \
\ssub{\|a\|}{[k+1]}=\sum\limits_{j=0}^\infty \ssub{\|b_j\|}{[k]}.
\end{equation}
By the definitions of $\Phi_\xi$ and $\Phi_{\xi'}$, we get
\begin{equation}\label{PHI}
p=\Phi_{\xi'}(a)=\sum_{n=0}^\infty \Phi_\xi(b_n)\xi_{k+1}^n\ \ \
\text{in $\hh_{\gamma(\xi')}$}.
\end{equation}
By (\ref{ab}), the series $\sum b_n$ converges absolutely in the
Banach space $\aa^{[k]}$. Since $\Phi_\xi:\aa^{[k]}\to\hh_\gamma$ is
a continuous linear operator, the series $\sum \Phi_{\xi}(b_n)$
converges absolutely in the Fr\'echet space $\hh_{\gamma(\xi)}$ and
therefore in the Fr\'echet space $\hh_U$. Since
$\xi_{k+1}(U)\subseteq\D$, the series in (\ref{PHI}) converges in
$\hh_U$. Since $U$ is open, connected and contains $0$, the sum of
the series in (\ref{PHI}) and $p$ coincide as functions on $U$ by
the uniqueness theorem: they are both holomorphic on $U$ and have
the same Taylor series at $0$. Since $q\big|\xi_{k+1}$, (\ref{PHI})
implies that $q\big|(p-\Phi_\xi(b_0))$ in $\hh_U$. Since all zeros
of $q$ are in $U$, $q\big|(p-\Phi_\xi(b_0))$ in $\hh_{\gamma(\xi)}$.
By (\ref{piuq}) and (\ref{ab}), $\pi_{\xi,q}(p)\leq
\ssub{\|b_0\|}{[k]}\leq \ssub{\|a\|}{[k+1]}$. Since $a$ is an
arbitrary element of $\aa^{[k+1]}$ satisfying $\Phi_{\xi'}(a)=p$,
(\ref{piu}) implies that $\pi_{\xi,q}(p)\leq \pi_{\xi'}(p)$.
\end{proof}

\begin{lemma}\label{techn} Let $q\in\pp_0$, $n\in\N$ and $k>0$ be
such that $\deg q<n$. For every $c>0$, let $\delta(c)=(2kc)^{-1/n}$
and $q_c=k(cz^n-q)\in\pp_0$. Then for every sufficiently large
$c>0$, $q_c(\delta(c)\D)\subseteq \D$ and all zeros of $q_c$ belong
to $\delta(c)\D$.
\end{lemma}

\begin{proof} Obviously, $\lim\limits_{c\to\infty}\delta(c)=0$.
Since $q(0)=0$, there is $\alpha>0$ such that $|q(z)|\leq \alpha|z|$
for all $z\in\D$. Clearly, it suffices to show that
$q_c(\delta(c)\D)\subseteq \D$ and all zeros of $q_c$ belong to
$\delta(c)\D$ whenever $\delta(c)<\min\{1,\frac1{2k\alpha}\}$.

Let $c>0$ be such that $\delta(c)<\min\{1,\frac1{2k\alpha}\}$. If
$z\in\delta(c)\D$, then
$|kcz^n|<kc\delta(c)^n=\frac{kc}{2kc}=\frac12$ and $|kq(z)|\leq
k\alpha\delta(c)<\frac{k\alpha}{2k\alpha}=\frac12$. Hence
$|q_c(z)|\leq |kcz^n|+|kq(z)|<\frac12+\frac12=1$. Thus
$q_c(\delta(c)\D)\subseteq \D$.

Now if $|z|=\delta(c)$, then
$|kcz^n|=kc\delta(c)^n=\frac{kc}{2kc}=\frac12$, but $|kq(z)|\leq
k\alpha\delta(c)<\frac{k\alpha}{2k\alpha}=\frac12$. By the Rouch\'e
theorem \cite {rou}, $q_c=kcz^n-kq$ has the same number of zeros
(counting with multiplicity) in $\delta(c)\D$ as $kcz^n$. The latter
has $n=\deg q_c$ zeros in $\delta(c)\D$. Hence all the zeros of
$q_c$ are in $\delta(c)\D$.
\end{proof}

The proof of the next lemma is postponed until further sections.

\begin{lemma}\label{techn-terr} Let $k,\delta>0$, $p\in\pp\setminus\{0\}$
and $m\in\N$. Then for every sufficiently large $n\in\N$, there
exists a connected open set $W_n\subset \C$ such that $0\in
W_n\subseteq\delta\D$ and the polynomial $q_n=kz((1+z)^n-p)$ has at
least $m$ zeros $($counting with multiplicity$)$ in $W_n$ and
satisfies $q_n(W_n)\subseteq\D$.
\end{lemma}

\begin{corollary}\label{techn-terr1} Let $k>0$, $p\in\pp\setminus\{0\}$
and $m\in\N$. Then there is $n_0\in\N$ and sequences $\{W_n\}_{n\geq
n_0}$ of connected non-empty open subsets of $\C$ containing $0$ and
$\{r_n\}_{n\geq n_0}$ of degree $m$ polynomials such that $r_n\to
z^m$, $\lim\limits_{n\to\infty}\sup\limits_{z\in W_n}|z|=0$, each
$r_n$ is a divisor of $q_n=kz((1+z)^n-p)$, $q_n(W_n)\subseteq\D$ and
all zeros of $r_n$ are in $W_n$ for each $n\geq n_0$.
\end{corollary}

\begin{proof} Applying Lemma~\ref{techn-terr} with $\delta=2^{-k}$
for $k\in\Z_+$, we find a strictly increasing sequence
$\{n_k\}_{k\in\Z_+}$ of positive integers such that for every
$k\in\Z_+$ and every $n\geq n_k$, there is a connected open subset
$W_{k,n}$ of $\C$ for which
\begin{equation}\label{wkn}
\begin{array}{l}
\text{$0\in W_{k,n}\subseteq 2^{-k}\D$, $q_n(W_{k,n})\subseteq\D$
and}\\ \text{$q_n$ has at least $m$ zeros in $W_{k,n}$ for every
$k\in\Z_+$ and $n\geq n_k$}.\end{array}
\end{equation}
The latter means that we can pick
$\lambda_{k,n,1},\dots,\lambda_{k,n,m}\in W_{k,n}$ such that
$r_{k,n}=\prod\limits_{j=1}^m (z-\lambda_{k,n,j})$ is a divisor of
$q_n$. Now for every $n\geq n_0$, we define $r_n=r_{k,n}$ and
$W_n=W_{k,n}$ whenever $n_k\leq n<n_{k+1}$. According to
(\ref{wkn}), each $r_n$ is a divisor of $q_n$, each $r_n$ has all
its zeros in $W_n$, $q_n(W_n)\subseteq \D$ and $W_n\subseteq
2^{-k}\D$ provided $n_k\leq n<n_{k+1}$. The latter means that
$\lim\limits_{n\to\infty}\sup\limits_{z\in W_n}|z|=0$ and also that
$r_n\to z^m$.
\end{proof}

\subsection{Proof of Theorem~\ref{main} modulo
Lemmas~\ref{wstarH} and~\ref{techn-terr}}

Now we take Lemmas~\ref{wstarH} and~\ref{techn-terr} as granted and
prove Theorem~\ref{main}. Fix a sequence $\{p_n\}_{n\in\N}$ in
$\pp_0\setminus\{0\}$ satisfying (\ref{ppp1}). We describe an
inductive procedure of constructing sequences $\{\xi_k\}_{k\in\N}$
in $\pp_0$, $\{n_k\}_{k\in\N}$ of natural numbers and
$\{c_{2k}\}_{k\in\N}$ of positive numbers such that
\begin{itemize}\itemsep=-2pt
\item[(A0)]$\xi_1=z$ and $n_1=1$;
\item[(A1)]$\pi_{\xi_{[k]}}(z)>\frac12$ for each $k\in\N$, where
$\xi_{[k]}=(\xi_1,\dots,\xi_k)\in\pp_0^k$;
\item[(A2)]$n_k>n_{k-1}$ for $k\geq 2$;
\item[(A3)]$\xi_k=k(c_{k}z^{n_k}-p_{k/2})$ for even $k\geq 2$ and
$\xi_k=k(z(1+z)^{n_k}-p_{(k-1)/2})$ for odd $k\geq 3$.
\end{itemize}

First, we take $n_1=1$, $\xi_1=z$ and observe that
$\pi_{\xi_{[1]}}(a_0+a_1z+{\dots}+a_mz^m)=|a_0|+{\dots}+|a_m|$. In
particular, $\pi_{\xi_{[1]}}(z)=1>\frac12$. Thus (A0--A3) for $k=1$
are satisfied and we have got the basis of induction. It remains to
describe the induction step. Let $k\geq 2$ and $\xi_j$, $n_j$ for
$j<k$ and $c_j$ for $j<k$ satisfying (A0--A3) are already
constructed. We shall construct $\xi_k$, $n_k$ and $c_k$ (if $k$ is
even), satisfying (A1--A3).

Denote $\gamma=\gamma(\xi_{[k-1]})$. By Lemma~\ref{wstarH},
$\pi_{\xi_{[k-1]},z^n}(z)\to \pi_{\xi_{[k-1]}}(z)$ as $n\to\infty$.
By (A1) for $k-1$, $\pi_{\xi_{[k-1]}}(z)>\frac12$. Hence we can pick
$m\in\N$ such that
\begin{equation}\label{l0}
\text{$\textstyle \pi_{\xi_{[k-1]},z^n}(z)>\frac12$ for every $n\geq
m$.}
\end{equation}

{\bf Case~1:} $k$ is even. By (\ref{l0}), there is $n_k\in\N$ such
that $n_k>\max\{n_{k-1},\deg p_{k/2}\}$ and
$\pi_{\xi_{[k-1]},z^{n_k}}(z)>\frac12$. For $c>0$, we consider the
degree $n_k$ polynomial $q_c=k(cz^{n_k}-p_{k/2})\in\pp_0$ and denote
$\delta(c)=(2kc)^{-1/{n_k}}$. Clearly, $\delta(c)\to 0$ as
$c\to\infty$. By Lemma~\ref{techn},
\begin{equation}\label{NNN}
\text{$\delta(c)<\gamma$, $q_c(\delta(c)\D)\subseteq \D$ and all
zeros of $q_c$ are in $\delta(c)\D$ for all sufficiently large
$c>0$.}
\end{equation}
Since $\frac1{kc}q_c=z^{n_k}-\frac1{c}p_{k-1}\to z^{n_k}$ as
$c\to\infty$, Lemma~\ref{wstarH} implies that
\begin{equation}\label{NNNN}
\pi_{\xi_{[k-1]},q_c}(p)=\pi_{\xi_{[k-1]},\frac1{kc}q_c}(p)\to
\pi_{\xi_{[k-1]},z^{n_k}}(p)\ \ \text{as $c\to\infty$ for every
$p\in\pp$.}
\end{equation}
Using (\ref{NNNN}), (\ref{NNN}) and the inequality
$\pi_{\xi_{[k-1]}}(z)>\frac12$, we can choose $c_k>0$ large enough
in such a way that $\delta=\delta(c_k)<\gamma$, all zeros of
$\xi_k=q_{c_k}=k(c_kz^{n_k}-p_{k/2})$ are in $\delta\D$,
$\xi_k(\delta\D)\subseteq \D$ and
$\pi_{\xi_{[k-1]},\xi_{k}}(z)>\frac12$. By Lemma~\ref{kaka},
$\pi_{\xi_{[k]}}(p)\geq \pi_{\xi_{[k-1]},\xi_k}(p)$ for every
$p\in\pp$. In particular, $\pi_{\xi_{[k]}}(z)\geq
\pi_{\xi_{[k-1]},\xi_k}(z)>\frac12$. It remains to notice that
(A1--A3) are satisfied.

{\bf Case~2:} $k$ is odd. By (\ref{l0}),
$\pi_{\xi_{[k-1]},z^m}(z)>\frac12$. By Corollary~\ref{techn-terr1},
there is $l\in\N$ and sequences $\{W_n\}_{n\geq l}$ of connected
non-empty open subsets of $\C$ containing $0$ and $\{r_n\}_{n\geq
l}$ of degree $m$ polynomials such that $r_n\to z^m$,
$\lim\limits_{n\to\infty}\sup\limits_{z\in W_n}|z|=0$, each $r_n$ is
a divisor of $q_n=k(z(1+z)^n-p_{(k-1)/2})$, $q_n(W_n)\subseteq\D$
and all zeros of $r_n$ are in $W_n$ for each $n\in\N$. By
Lemma~\ref{wstarH}, $\pi_{\xi_{[k-1]},r_n}(z)\to
\pi_{\xi_{[k-1]},z^m}(z)>\frac12$ as $n\to\infty$ and therefore we
can pick $n_k>\max\{l,n_{k-1}\}$ such that
$\pi_{\xi_{[k-1]},r_{n_k}}(z)>\frac12$ and $W_{n_k}\subseteq
\gamma\D$. Put $\xi_k=q_{n_k}=k(z(1+z)^{n_k}-p_{(k-1)/2})$. By
Lemma~\ref{kaka}, $\pi_{\xi_{[k]}}(z)\geq
\pi_{\xi_{[k-1]},r_{n_k}}(z)>\frac12$. It remains to notice that
(A1--A3) are again satisfied.

This concludes the inductive construction of the sequences
$\{\xi_k\}_{k\in\N}$, $\{n_k\}_{k\in\N}$ and $\{c_{2k}\}_{k\in\N}$
satisfying (A0--A3). By Lemma~\ref{kaka}, $\pi_{\xi_{[k+1]}}(p)\leq
\pi_{\xi_{[k]}}(p)$  for every $p\in\pp$. Thus,
$\{\pi_{\xi_{[k]}}\}_{k\in\N}$ is a pointwise decreasing sequence of
submultiplicative norms on $\pp$. Hence the formula
$\pi(p)=\lim\limits_{k\to\infty}\pi_{\xi_{[k]}}(p)$ defines a
submultiplicative seminorm on $\pp$. By (A1),
$\pi_{\xi_{[k]}}(z)>\frac12$ for each $k\in\N$ and therefore
$\pi(z)\geq\frac12>0$. Hence $\pi$ is non-zero. From (\ref{piu}) it
immediately follows that $\pi_{\xi_{[k]}}(\xi_k)\leq 1$ for every
$k\in\N$. Indeed, $\ssub{\|u_k\|}{[k]}=1$ and
$\Phi_{\xi_{[k]}}(u_k)=\xi_k$. Hence $\pi(\xi_k)\leq
\pi_{\xi_{[k]}}(\xi_k)\leq 1$. By (A3),
$\xi_{2k}=2k(c_{2k}z^{n_{2k}}-p_{k})$ for $k\in\N$. Hence
$\pi(c_{2k}z^{n_{2k}}-p_{k})\leq \frac1{2k}$ for every $k\in\N$ and
therefore $\pi(c_{2k}z^{n_{2k}}-p_{k})\to 0$. By (A3),
$\xi_{2k+1}=(2k+1)(z(1+z)^{n_{2k+1}}-p_{k})$ for $k\in\N$. Hence
$\pi(z(1+z)^{n_{2k+1}}-p_{k})\leq \frac1{2k+1}$ for every $k\in\N$
and therefore $\pi(z(1+z)^{n_{2k+1}}-p_{k})\to 0$. Thus all
conditions of Lemma~\ref{z2} are satisfied. By Lemma~\ref{z2}, the
restriction of $\pi$ to $\pp_0$ is a submultiplicative norm on
$\pp_0$ and the completion of the normed algebra $(\pp_0,\pi)$ is an
infinite dimensional chaotic Banach algebra with $z$ being a chaotic
element. The proof of Theorem~\ref{main} modulo Lemma~\ref{wstarH}
and~\ref{techn-terr} is complete.

\section{Proof of Lemma~\ref{wstarH}}

\begin{lemma}\label{wsn} Let $X$ be a Banach space and
$\{L_n\}_{n\in\N}$ be a decreasing sequence of $*$-weak closed
linear subspaces of $X^*$ and $L=\bigcap\limits_{n=1}^\infty L_n$.
Then for every $g\in X^*$,
$\ssub{\|g+L_n\|}{X^*/L_n}\to\ssub{\|g+L\|}{X^*/L}$.
\end{lemma}

\begin{proof} Let $g\in X^*$. Since $\ssub{\|g+L_n\|}{X^*/L_n}\leq
\ssub{\|g+L_{n+1}\|}{X^*/L_{n+1}} \leq \ssub{\|g+L\|}{X^*/L}$ for
each $n\in\N$, the sequence $\ssub{\|g+L_n\|}{X^*/L_n}$ converges to
$c\in\R_+$ and $c\leq c_1=\ssub{\|g+L\|}{X^*/L}$. It remains to show
that $c\geq c_1$.

By definition of the quotient norms, we can find $f_n\in L_n$ for
$n\in\N$ such that $\ssub{\|g+f_n\|}{X^*}\to c$. Since $\{f_n\}$ is
a bounded sequence in $X^*$, and closed balls in $X^*$ are $*$-weak
compact, there is a $*$-weak accumulation point $f$ of the sequence
$\{f_n\}$. Since $f_m\in L_n$ for $m\geq n$ and each $L_n$ is
$*$-weak closed, $f\in L_n$ for every $n\in\N$. That is, $f\in L$.
Since $\ssub{\|g+f_n\|}{X^*}\to c$, the ball $\{h\in X^*:\|h\|\leq
c\}$ is $*$-weak compact and $g+f$ is a $*$-weak accumulation point
of $\{g+f_n\}$, we get $\ssub{\|g+f\|}{X^*}\leq c$. Since $f\in L$,
$c_1=\ssub{\|g+L\|}{X^*/L}\leq \ssub{\|g+f\|}{X^*}\leq c$, which
completes the proof.
\end{proof}

Note that the same statement with $*$-weak closeness replaced by the
norm closeness is false.

\begin{lemma}\label{quo} Let $m\in\N$, $X$ be a normed space,
$z_j\in X$ and $f_j,f_{j,k}\in X^*$ be such that
$f_j(z_r)=f_{j,k}(z_r)=\delta_{j,r}$ for $1\leq j,r\leq m$ and
$k\in\N$ and $\lim\limits_{k\to\infty}\|f_{j,k}-f_k\|=0$ for $1\leq
j\leq m$. Then
\begin{equation}\label{bubu}
\ssub{\|x+Y_k\|}{X/Y_k}\to \ssub{\|x+Y\|}{X/Y}\ \ \text{for each
$x\in X$, where $Y=\bigcap\limits_{j=1}^m \ker f_j$ and
$Y_k=\bigcap\limits_{j=1}^m \ker f_{j,k}$ for $k\in\N$.}
\end{equation}
\end{lemma}

\begin{proof} Let $x\in X$. Then we can pick two sequences
$\{u_k\}_{k\in\N}$ and $\{w_k\}_{k\in\N}$ in $X$ such that
\begin{equation}\label{bubu1}
\text{$u_k\in Y$ and $w_k\in Y_k$ for $k\in\N$,
$\|x+u_k\|\to\ssub{\|x+Y\|}{X/Y}$ and
$(\|x+w_k\|-\ssub{\|x+Y_k\|}{X/Y_k})\to 0$.}
\end{equation}
The inequalities $\ssub{\|x+Y_k\|}{X/Y_k}\leq \|x\|$ together with
(\ref{bubu1}) imply that $\{u_k\}_{k\in\N}$ and $\{w_k\}_{k\in\N}$
are bounded: $\|u_k\|\leq C$ and $\|w_k\|\leq C$ for every $k\in\N$,
where $C$ is a positive constant. Consider
$$
u'_k=u_k-\sum_{j=1}^n f_{j,k}(u_k)z_j\ \ \ \text{and}\ \ \
w'_k=w_k-\sum_{j=1}^n f_{j}(w_k)z_j\ \ \ \text{for $k\in\N$}.
$$
Since $f_j(z_r)=f_{j,k}(z_r)=\delta_{j,r}$,
$f_{j,k}(u'_k)=f_{j}(w'_k)=0$ for $1\leq j\leq m$ and $k\in\N$,
$u'_k\in Y_k$ and $w'_k\in Y$ for each $k\in\N$. Since $u_k\in Y$
and $w_k\in Y_k$, $f_j(u_k)=f_{j,k}(w_k)=0$ for $1\leq j\leq m$ and
$k\in\N$. Hence
$$
|f_{j,k}(u_k)|=|(f_{j,k}-f_j)(u_k)|\leq C\|f_{j,k}-f_j\|\to 0\ \
\text{and}\ \ |f_{j}(w_k)|=|(f_{j}-f_{j,k})(w_k)|\leq
C\|f_{j,k}-f_j\|\to 0.
$$
By the above two displays, $\|u'_k-u_k\|\to 0$ and $\|w'_k-w_k\|\to
0$. Since $u'_k\in Y_k$, $\ssub{\|x+Y_k\|}{X/Y_k}\leq \|x+u'_k\|$.
According to (\ref{bubu1}),
$$
\ssub{\|x+Y\|}{X/Y}=\lim_{k\to\infty}\|x+u_k\|=\lim_{k\to\infty}\|x+u'_k\|\geq
\slim_{k\to \infty} \ssub{\|x+Y_k\|}{X/Y_k}.
$$
Since $w'_k\in Y$, $\ssub{\|x+Y\|}{X/Y}\leq \|x+w'_k\|$. Using
(\ref{bubu1}), we get
$$
\ssub{\|x+Y\|}{X/Y}\leq\ilim_{k\to\infty}
\|x+w'_k\|=\ilim_{k\to\infty}\|x+w_k\|=\ilim_{k\to\infty}\ssub{\|x+Y_k\|}{X/Y_k}.
$$
The required equality (\ref{bubu}) follows from the last two
displays.
\end{proof}

The next bunch of facts, we need, is about Van-der-Monde-like
determinants. The Van-der-Monde matrix is the matrix of the shape
$M_n(\lambda)=\{\lambda_j^{r-1}\}_{j,r=1}^n$ with
$\lambda=(\lambda_1,\dots,\lambda_n)\in\C^n$:
$$
M_n(\lambda)=\begin{pmatrix}
1&\lambda_1&\lambda_1^2&\dots&\lambda_1^{n-1}\\
1&\lambda_2&\lambda_2^2&\dots&\lambda_2^{n-1}\\
\vdots&\vdots&\vdots&\ddots&\vdots\\
1&\lambda_n&\lambda_n^2&\dots&\lambda_n^{n-1}
\end{pmatrix}.
$$
It is well-known \cite{vdm}, that the determinant of $M_n(\lambda)$
is given by the formula
\begin{equation}\label{vdmdet}
\det M_n(\lambda)=\prod_{1\leq j< r\leq n}(\lambda_r-\lambda_j).
\end{equation}
Clearly $\det M_n(\lambda)$ is a non-zero homogeneous polynomial in
the variables $\lambda_1,\dots,\lambda_n$ of degree
$\frac{n(n-1)}{2}$ (a homogeneous polynomial of degree $m$ is a
linear combination of monomials of degree $m$ and therefore the zero
polynomial is a homogeneous polynomial of any degree). We need the
related matrices $M_{n,j,m}(\lambda)$ with $1\leq j\leq n$ and
$m\in\Z_+$, obtained from $M_n(\lambda)$ by replacing the $j^{\rm
th}$ column by the column vector with $\{\lambda_r^m\}_{r=1}^n$:
$$
\text{$M_{n,j,m}(\lambda)$ is $M_n(\lambda)$ with the $j^{\rm th}$
column}\ \ \begin{pmatrix}\lambda_1^{j-1}\\ \vdots\\
\lambda_n^{j-1}\end{pmatrix}\ \ \text{replaced by the column}\ \
\begin{pmatrix}\lambda_1^{m}\\ \vdots\\
\lambda_n^{m}\end{pmatrix}.
$$
Clearly $\det M_{n,j,m}(\lambda)\in \C[\lambda_1,\dots,\lambda_n]$.
Since a permutation of two rows multiplies the determinant of a
matrix by $-1$, the polynomial $\det M_{n,j,m}(\lambda)$ is
antisymmetric. It is easy to see that $\det M_{n,j,m}(\lambda)$ is a
homogeneous polynomial of degree $\frac{n(n-1)}{2}+m-j+1$. Observe
that $\det M_{n,j,m}(\lambda)=0$ if $\lambda_j=\lambda_r$ for some
$1\leq j<r\leq n$. Indeed, in this case $M_{n,j,m}(\lambda)$ has two
identical rows. Hence $\lambda_r-\lambda_j$ is a divisor of $\det
M_{n,j,m}(\lambda)$ for $1\leq j<r\leq n$ and therefore, by
(\ref{vdmdet}), $\det M_n(\lambda)$ is a divisor of $\det
M_{n,j,m}(\lambda)$. Thus
\begin{equation}\label{pnjm}
P_{n,j,m}(\lambda)=\frac{\det M_{n,j,m}(\lambda)}{\det
M_{n}(\lambda)}\in \C[\lambda_1,\dots,\lambda_n]
\end{equation}
and the polynomial $P_{n,j,m}$ is symmetric since both $\det
M_{n,j,m}(\lambda)$ and $\det M_{n}(\lambda)$ are antisymmetric.
Since $M_{n,j,j-1}(\lambda)=M_n(\lambda)$, $P_{n,j,j-1}(\lambda)=1$
for $1\leq j\leq n$. If $0\leq m\leq n-1$ and $m\neq j-1$, then the
matrix $M_{n,j,m}(\lambda)$ has two identical columns and therefore
$\det M_{n,j,m}(\lambda)=0$. Hence,
\begin{equation}\label{basis}
P_{n,j,m}(\lambda)=\delta_{m,j-1}\quad \text{for $0\leq m\leq n-1$}.
\end{equation}
Consider the following two matrices with $n$ rows and infinitely
many columns $A_n(\lambda)=\{\lambda_j^{r-1}\}_{1\leq j\leq n,\
r\in\N}$ and $B_n(\lambda)=\{P_{n,j,r-1}(\lambda)\}_{1\leq j\leq n,\
r\in\N}$:
$$
B_n(\lambda)=\begin{pmatrix}
P_{n,1,0}(\lambda)&P_{n,1,1}(\lambda)&P_{n,1,2}(\lambda)&\dots\\
P_{n,2,0}(\lambda)&P_{n,2,1}(\lambda)&P_{n,2,2}(\lambda)&\dots\\
\vdots&\vdots&\vdots&\vdots\\
P_{n,n,0}(\lambda)&P_{n,n,1}(\lambda)&P_{n,n,2}(\lambda)&\dots
\end{pmatrix}\quad\text{and}\quad
A_n(\lambda)=\begin{pmatrix}
1&\lambda_1&\lambda_1^2&\lambda_1^3&\dots\\
1&\lambda_2&\lambda_2^2&\lambda_2^3&\dots\\
\vdots&\vdots&\vdots&\vdots&\vdots\\
1&\lambda_n&\lambda_n^2&\lambda_n^3&\dots
\end{pmatrix}.
$$
The way we have defined $P_{n,j,m}(\lambda)$ together with an
elementary linear algebra exercise yields
\begin{equation}\label{prod}
M_n(\lambda)B_n(\lambda)=A_n(\lambda).
\end{equation}
We need upper estimates for the values of $P_{n,j,m}$. For $m\geq
n$, we do not have an explicit formula for $P_{n,j,m}$, but we can
find a recurrent formula, which allows to obtain the necessary
estimates.

\begin{lemma}\label{estimp} Let $n\in\N$ and $0<\alpha<\beta$.
Then there exists $c=c(n,\alpha,\beta)>0$ such that
\begin{equation}\label{estp}
\text{$|P_{n,j,m}(\lambda)|\leq c\beta^m$ whenever $m\in\Z_+$,
$1\leq j\leq n$ and $\max\limits_{1\leq r\leq
n}|\lambda_r|\leq\alpha$.}
\end{equation}
\end{lemma}

\begin{proof} We use the induction with respect to $n$. For $n=1$,
$P_{1,1,m}(\lambda)=\lambda_1^m$ and therefore
$|P_{1,1,m}(\lambda)|\leq \alpha^m\leq\beta^m$ for each $m\in\Z_+$
and $\lambda_1\in\alpha\DD$. Thus (\ref{estp}) is satisfied for
$n=1$ with $c(1,\alpha,\beta)=1$.

In order to run the inductive proof, we need recurrent formulas for
$P_{n,j,m}$. If $j>1$, the first column of $M_{n,j,m}(\lambda)$
consist of $1$'s. Subtracting the first row from all other rows, and
eliminating the first row and column of the resulting matrix, we get
an $(n-1)\times(n-1)$ matrix, whose determinant is the same as for
$M_{n,j,m}(\lambda)$. Then after dividing the $r^{\rm th}$ row of
the new matrix by $\lambda_{r+1}-\lambda_1$, doing a number of
determinant preserving manipulations with the columns and using the
column-linearity of the determinant, we get
\begin{align}\label{rec1}
&P_{n,j,m}(\lambda)=\sum_{s=0}^{m-n}
(\lambda_1^sP_{n-1,j-1,m-s-1}(\lambda')-
\lambda_1^{s+1}P_{n-1,j,m-s-1}(\lambda'))\ \ \ \text{if $1<j<n$ and
$m\geq n\geq3$,}
\\ \label{rec2}
&P_{n,n,m}(\lambda)=\lambda_1^{m-n+1}+
\sum_{s=0}^{m-n}\lambda_1^sP_{n-1,n-1,m-s-1}(\lambda')\ \ \ \text{if
$m\geq n\geq2$, \ where
$\lambda'=(\lambda_2,\lambda_3,\dots,\lambda_n)$.}
\end{align}
If $j=1$, we can divide the $r^{\rm th}$ row of $M_{n,1,m}(\lambda)$
by $\lambda_r$ for $1\leq r\leq n$ and then exchange the first and
the last columns to obtain the equality
$$
P_{n,1,m}(\lambda)=-\lambda_1\dots\lambda_nP_{n,n,m-1}(\lambda)\ \ \
\text{for $m\geq n\geq 2$}.
$$
Using (\ref{rec2}) and (\ref{basis}), we get
\begin{equation}\label{rec3}
\begin{array}{l}
\text{$P_{n,1,n}(\lambda)=-\lambda_1\dots\lambda_n$\ \ and}
\\
\displaystyle
P_{n,1,m}(\lambda)=-\lambda_1\dots\lambda_n\Bigl(\lambda_1^{m-n}+
\sum\limits_{s=0}^{m-n-1}\lambda_1^sP_{n-1,n-1,m-s-1}(\lambda')\Bigr)\
\ \ \text{for $m>n\geq 2$}.
\end{array}
\end{equation}

Now we assume that $n\geq 2$ and that the required estimate holds
for smaller values of $n$. Let $0<\alpha<\beta$. Pick $\gamma\in
(\alpha,\beta)$. By the induction hypothesis, there is $c_0>0$ such
that
\begin{align}\label{hyp}
&\text{$|P_{n-1,j,m}(w)|\leq c_0\gamma^m$ for any $m\in\Z_+$, $1\leq
j\leq n-1$ and $w\in(\alpha\DD)^{n-1}$.}
\\ \label{smallm}
&\text{By (\ref{basis}), $|P_{n,j,m}(\lambda)|\leq 1$\ \ \ for
$1\leq j\leq n$, $0\leq m\leq n-1$ and $\lambda\in(\alpha\DD)^n$.}
\end{align}
Using (\ref{rec1}), (\ref{rec2}) and (\ref{rec3}) together with
(\ref{hyp}) we find that there is $a=a(n,\alpha,\gamma)>0$ such that
\begin{equation}\label{bigm}
|P_{n,j,m}(\lambda)|\leq a m\gamma^m\ \ \ \text{for $1\leq j\leq n$,
$m\geq n$ and $\lambda\in(\alpha\DD)^n$.}
\end{equation}
Since $\frac{m\gamma^m}{\beta^m}\to 0$, (\ref{smallm}) and
(\ref{bigm}) imply that (\ref{estp}) is satisfied with some
$c=c(n,\alpha,\beta)>0$.
\end{proof}

\begin{lemma}\label{conti} Let $k,n\in\N$,
$\xi=(\xi_1,\dots,\xi_k)\in\pp_0^k$, $\xi_1=z$ and
$\gamma=\gamma(\xi)$. Then for every $\lambda\in(\gamma\D)^n$ and
$1\leq j\leq n$, the formula
\begin{equation}\label{fun}
\phi_{\lambda,j}:\aa^{[k]}\to\C, \quad
\phi_{\lambda,j}(a)=\sum_{m=0}^\infty
P_{n,j,m}(\lambda)\alpha_m(a,\xi)
\end{equation}
defines a continuous linear functional on $\aa^{[k]}$, where where
$\alpha_m(a,\xi)$ are defined in $(\ref{tra})$. Moreover,
\begin{align}\label{fun1}
&\phi_{\lambda,j}(u_1^{r-1})=\delta_{j,r}\ \ \ \text{for $1\leq
j,r\leq n$ and}
\\ \label{fun2}
&I_{\xi,q_\lambda}=\bigcap_{j=1}^n\ker\phi_{\lambda,j},\ \ \
\text{where $\textstyle
q_{\lambda}(z)=\prod\limits_{j=1}^n(z-\lambda_j)$ and $\textstyle
I_{\xi,q_\lambda}$ is defined by $(\ref{ietq})$.}
\end{align}
Furthermore, the map $\lambda\mapsto \phi_{\lambda,j}$ from
$(\gamma\D)^n$ to $\bigl(\aa^{[k]}\bigr)^*$ is norm continuous for
$1\leq j\leq n$.
\end{lemma}

\begin{proof} For $1\leq j\leq n$ and $\lambda\in(\gamma\D)^n$, consider
the functionals $\psi_{\lambda,j}:\hh_\gamma\to\C$ defined by the
formula
$$
\psi_{\lambda,j}\Bigl(\sum_{m=0}^\infty
f_mz^m\Bigr)=\sum_{m=0}^\infty P_{n,j,m}(\lambda)f_m.
$$
By Lemma~\ref{estimp}, $P_{n,j,m}(\lambda)=o(\beta^m)$ as
$m\to\infty$ for $1\leq j\leq n$ for some
$\beta=\beta(\lambda)<\gamma$. It follows that the functionals
$\psi_{\lambda,j}$ are well-defined and continuous. We already know
that $\Phi_\xi:\aa^{[k]}\to\hh_{\gamma}$ is a continuous algebra
homomorphism. Since $\phi_{\lambda,j}=\psi_{\lambda,j}\circ
\Phi_\xi$ and $\Phi:\aa^{[k]}\to\hh_\gamma$ is continuous, the
functionals $\phi_{\lambda,j}$ are also well-defined and continuous.
From the equality $\xi_1=z$ and the definition of $\phi_{\lambda,j}$
it follows that $\phi_{\lambda,j}(u_1^r)=P_{n,j,r}(\lambda)$ for
$r\geq 0$. Hence (\ref{basis}) implies (\ref{fun1}).

Let $J_\lambda$ be the ideal in $\hh_\gamma$, generated by
$q_\lambda$. Since $J_\lambda$ is a closed ideal in $\hh_\gamma$ of
codimension $n=\deg q_\lambda$ and
$I_{\xi,q_\lambda}=\Phi_\xi^{-1}(J_\lambda)$, the ideal
$I_{\xi,q_\lambda}$ is closed in $\aa^{[k]}$ and has codimension
$\leq n$. Since $u_1=z$, we see that $1,u_1,\dots,u_1^{n-1}$ are
linearly independent modulo $I_{\xi,q}$. Hence the codimension of
$I_{\xi,q_\lambda}$ in $\aa^{[k]}$ is exactly $n$. By (\ref{fun1}),
the functionals $\phi_{\lambda,1},\dots,\phi_{\lambda,n}$ are
linearly independent and therefore (\ref{fun2}) will be verified if
we prove that each $\phi_{\lambda,j}$ vanishes on
$I_{\xi,q_\lambda}$. Since $\phi_{\lambda,j}=\psi_{\lambda,j}\circ
\Phi_\xi$ and $I_{\xi,q_\lambda}=\Phi_\xi^{-1}(J_\lambda)$, it is
enough to demonstrate that each $\psi_{\lambda,j}$ vanishes on
$J_\lambda$. That is, it suffices to show that
$\psi_{\lambda,j}(fq_\lambda)=0$ for any $f\in\hh_\gamma$, $1\leq
j\leq n$ and $\lambda\in(\gamma\D)^n$. Since $\pp$ is dense in
$\hh_\gamma$, it is enough to consider $f\in\pp$. Fix $1\leq j\leq
n$ and $f\in\pp$. We have to verify that
$\psi_{\lambda,j}(fq_\lambda)=0$ for each $\lambda\in(\gamma\D)^n$.
Since $P_{n,j,m}(\lambda)\in\C[\lambda_1,\dots,\lambda_n]$,
$\lambda\mapsto \psi_{\lambda,j}(fq_\lambda)$ is a polynomial and it
suffices to verify the equality $\psi_{\lambda,j}(fq_\lambda)=0$ for
$\lambda$ from a dense subset of $(\gamma\D)^n$. Hence (\ref{fun2})
will be proved if we show that
$$
\text{$\psi_{\lambda,j}(fq_\lambda)=0$ for any pairwise different
$\lambda_1,\dots,\lambda_n$ in $\gamma\D$.}
$$
Let $\lambda_1,\dots,\lambda_n\in\gamma\D$ be pairwise different. By
(\ref{vdmdet}), the Van-der-Monde matrix $M_n(\lambda)$ is
invertible. Then (\ref{prod}) can be rewritten as
$B_n(\lambda)=M_n(\lambda)^{-1}A_n(\lambda)$. Hence each row of
$B_n(\lambda)$ is a linear combination of rows of $A_n(\lambda)$.
Thus there exist $c_1,\dots,c_n\in\N$ such that the sequence
$\{P_{n,j,r}(\lambda)\}_{r\in\Z_+}$ is the linear combination of the
sequences
$\{\lambda_1^r\}_{r\in\Z_+},\dots,\{\lambda_n^r\}_{r\in\Z_+}$ with
the coefficients $c_1,\dots,c_n\in \C$. It follows that
$$
\psi_{\lambda,j}(g)=\sum_{s=1}^n c_sg(\lambda_s)\ \ \text{for each
$g=\sum\limits_{m=0}^\infty g_nz^n\in\hh_\gamma$}.
$$
Since $(fq_\lambda)(\lambda_j)=0$ for $1\leq j\leq n$, the above
display implies that $\psi_{\lambda,j}(fq_\lambda)=0$, which
completes the proof of (\ref{fun2}). It remains to verify the norm
continuity of the maps $\lambda\mapsto \phi_{\lambda,j}$ from
$(\gamma\D)^n$ to $\bigl(\aa^{[k]}\bigr)^*$. Let
$\lambda\in(\gamma\D)^n$ and for each $s\in\N$,
$\lambda_s=(\lambda_{s,1},\dots,\lambda_{s,n})\in (\gamma\D)^n$ be
such that $\lambda_s\to\lambda$ in $\C^n$. We have to show that
$\|\phi_{\lambda_s,j}-\phi_{\lambda,j}\|\to 0$ as $s\to\infty$,
where the norm is taken in $\bigl(\aa^{[k]}\bigr)^*$. Pick
$\alpha\in (0,\gamma)$ such that $\lambda_{s,j}\in\alpha\DD$ and
$\lambda_j\in\alpha\DD$ for $1\leq j\leq n$ and $s\in\N$. Take any
$\beta\in(\alpha,\gamma)$. By Lemma~\ref{estimp}, there is $c>0$
such that
$$
\text{$|P_{n,j,m}(\lambda_s)|\leq c\beta^m$ and
$|P_{n,j,m}(\lambda)|\leq c\beta^m$ for $m\in\Z_+$ and $1\leq j\leq
n$.}
$$
Now let $K$ be the set of continuous linear functionals
$\psi:\hh_\gamma\to \C$ given by
$\smash{\psi(f)=\sum\limits_{m=0}^\infty \psi_mf_m}$, where
$f(z)=\sum\limits_{m=0}^\infty f_mz^m$, such that $|\psi_m|\leq
c\beta^m$ for $m\in\Z_+$. It is well-known and easy to see that $K$
is a compact subset of the dual space $\hh_\gamma^*$ equipped with
the strong topology \cite{rob}. Since
$\Phi_\xi:\aa^{[k]}\to\hh_\gamma$ is continuous, the dual operator
$\Phi_\xi^*:\hh^*_\gamma\to\bigl(\aa^{[k]}\bigr)^*$ is continuous
when both $\hh^*_\gamma$ and $\bigl(\aa^{[k]}\bigr)^*$ are equipped
with the strong topology \cite{rob}. Since $K$ is strongly compact
in $\hh^*_\gamma$ and the strong topology on the dual of a normed
space is the norm topology, $Q=\Phi^*_\xi(K)$ is norm compact in
$\bigl(\aa^{[k]}\bigr)^*$. By the above display,
$\psi_{\lambda_s,j}$ and $\psi_{\lambda,j}$ are all in $K$ and
therefore $\phi_{\lambda_s,j}=\Phi_\xi^*\psi_{\lambda_s,j}\in Q$ and
$\phi_{\lambda,j}=\Phi_\xi^*\psi_{\lambda,j}\in Q$. Since every,
$P_{n,j,m}(\lambda)$ depends polynomially and therefore continuously
on $\lambda$, it immediately follows that
$\phi_{\lambda_s,j}\to\phi_{\lambda,j}$ pointwise on the dense
subspace $\C[u_1,\dots,u_k]$ of $\aa^{[k]}$. Since the topology on
$\bigl(\aa^{[k]}\bigr)^*$ of pointwise convergence on
$\C[u_1,\dots,u_k]$ is Hausdorff and is weaker than the norm
topology, it must coincide with the norm topology on the norm
compact set $Q$. Since $\phi_{\lambda_s,j}\in Q$ and
$\phi_{\lambda,j}\in Q$, $\|\phi_{\lambda_s,j}-\phi_{\lambda,j}\|\to
0$ as $s\to\infty$ for $1\leq j\leq n$, as required.
\end{proof}

\begin{proof}[Proof of Lemma~\ref{wstarH}.] Recall that
$\xi=(\xi_1,\dots,\xi_k)\in\pp_0^k$, $\xi_1=z$ and for $q\in\pp$
with all zeros in $\gamma \D$ with $\gamma=\gamma(\xi)$, $I_{\xi,q}$
is the ideal in $\aa^{[k]}$ defined by (\ref{ietq}). By
Lemma~\ref{conti}, each $I_{\xi,q}$ is a closed ideal in $\aa^{[k]}$
of codimension $\deg q$. The ideal $I_{\xi}$ is also closed in
$\aa^{[k]}$ as the intersection of closed ideals $I_{\xi,z^n}$. The
inclusion $I_\xi\subseteq I_{\xi,q}$ is obvious.

Next, we observe that each $I_{\xi,z^n}$ is $*$-weak closed. Indeed,
by (\ref{ietn}), it is enough to show that each of the functionals
$a\mapsto\alpha_j(a,\xi)$ is $*$-weak continuous on $\aa^{[k]}$. The
latter is clear since each $\alpha_j(\cdot,\xi)$ is a finite linear
combination of the standard coordinate functionals on
$\aa^{[k]}=\ell_1(\Z_+^k)$. Since $I_{\xi,z^n}$ are $*$-weak closed,
Lemma~\ref{wsn} implies that
$\ssub{\|a+I_{\xi,z^n}\|}{A^{[k]}/I_{\xi,z^n}}\to
\ssub{\|a+I_{\xi}\|}{A^{[k]}/I_{\xi}}$ for every $a\in\aa^{[k]}$.
Thus (\ref{limi1}) is satisfied.

It remains to verify (\ref{limi2}). Let $q_n\in\pp$ for
$n\in\N\cup\{\infty\}$ be polynomials of degree $m\in\N$, whose
zeros are in $\gamma\D$ and the sequence $\{q_n\}_{n\in\N}$
converges to $q_\infty$. Without loss of generality, we may assume
that each $q_n$ is monic (=has the leading coefficient 1). Then,
taking into account that $q_n\to q_\infty$, we can write
$$
\text{$\textstyle q_\infty=\prod\limits_{j=1}^m(z-\lambda_{j})$ and
$\textstyle q_n=\prod\limits_{j=1}^m(z-\lambda_{j,n})$ with
$\lambda_{j},\lambda_{j,n}\in\gamma\D$ and
$\lambda_{j,n}\to\lambda_{j}$ as $n\to\infty$ for $1\leq j\leq m$.}
$$
Let $\lambda_n=(\lambda_{1,n},\dots,\lambda_{m,n})$ and
$\lambda=(\lambda_1,\dots,\lambda_m)$. Then $\lambda_n\to\lambda$ in
$\C^m$ and $\lambda$ and all $\lambda_n$ belong to $(\gamma\D)^m$.
Let $\phi_{\lambda_n,j}$ and $\phi_{\lambda,j}$ be the continuous
functionals on $\aa^{[k]}$ defined by (\ref{fun}). According to
Lemma~\ref{conti}, $\|\phi_{\lambda_n,j}-\phi_{\lambda,j}\|\to 0$ as
$n\to\infty$,
$\phi_{\lambda_n,j}(u_1^{r-1})=\phi_{\lambda,j}(u_1^{r-1})=\delta_{j,r}$
for $1\leq j,r\leq m$ and
$$
I_{\xi,q_\infty}=\bigcap_{j=1}^m \ker \phi_{\lambda,j}\ \
\text{and}\ \ I_{\xi,q_n}=\bigcap_{j=1}^m \ker \phi_{\lambda_n,j}\ \
\text{for $n\in\N$}.
$$
Now Lemma~\ref{quo} implies (\ref{limi2}). The proof of
Lemma~\ref{wstarH} is complete.
\end{proof}

\section{Proof of Lemma~\ref{techn-terr}}

Our main instrument is the argument principle \cite{rou}. We recall
the related basic concepts. An {\it oriented path} $\Gamma$ in $\C$
with the {\it source} $s(\Gamma)$ and the {\it end} $e(\Gamma)$ is a
set of the shape $\Gamma=\phi([a,b])$, where $\phi:[a,b]\to \C$ is
continuous, $\phi(a)=s(\Gamma)$, $\phi(b)=e(\Gamma)$ and
$\phi\bigr|_{(a,b)}$ is injective. Such a map $\phi$ is a {\it
parametrization} of the path $\Gamma$. The oriented path $\Gamma$ is
{\it closed} if $s(\Gamma)=e(\Gamma)$. If $\Gamma$ is an oriented
path in $\C$ and $f:\Gamma\to\C\setminus\{0\}$ is continuous, we can
find continuous $\phi:[a,b]\to\Gamma$ and $\psi:[a,b]\to\R$ such
that $\phi(a)=s(\Gamma)$, $\phi(b)=e(\Gamma)$ and
$\frac{f(\phi(t))}{|f(\phi(t))|}=e^{i\psi(t)}$ for every
$t\in[a,b]$. The number $\frac{\psi(b)-\psi(a)}{2\pi}$ does not
depend on the choice of $\phi$ and $\psi$ and is called the {\it
winding number of $f$ along the path $\Gamma$} and denoted
$w(f,\Gamma)$. Alternatively, $2\pi w(f,\Gamma)$ is the {\it
variation of the argument of $f$ along $\Gamma$}.

We need few well-known properties of the winding numbers. If
$\Gamma$ and $\Gamma'$ are two non-closed oriented paths with
$e(\Gamma)=s(\Gamma')$ and
$(\Gamma\setminus\{e(\Gamma),s(\Gamma)\})\cap
(\Gamma'\setminus\{e(\Gamma'),s(\Gamma')\})=\varnothing$, then
$\Gamma\cup\Gamma'$ can be naturally considered as an oriented path
with the source $s(\Gamma)$ and the end $e(\Gamma')$. Then
\begin{equation}\label{addi}
w(f,\Gamma\cup\Gamma')=w(f,\Gamma)+w(f,\Gamma')\ \ \text{for each
continuous $f:\Gamma\cup\Gamma'\to\C\setminus\{0\}$.}
\end{equation}

Variants of the following elementary property exist in the
literature under different names, one of which is the {\it dog on a
leash lemma}. If $\Gamma$ is an oriented path in $\C$ and
$f,g:\Gamma\to\C$ are continuous, then
\begin{equation}\label{appro}
|w(f+g,\Gamma)-w(f,\Gamma)|<1/2\ \ \ \text{if $|g(z)|<|f(z)|$ for
each $z\in\Gamma$.}
\end{equation}

It is easy to see that if $\Gamma$ is an oriented path,
$f:\Gamma\to\C\setminus\{0\}$ is continuous and $|w(f,\Gamma)|\geq
n/2$ with $n\in\N$, then $f$ crosses every line in $\C$ passing
through 0 at least $n$ times. In other words, if $c\in\T$, then
\begin{equation}\label{cros}
\textstyle |w(f,\Gamma)|<\frac{n+1}2\ \ \ \text{if
$\{z\in\Gamma:f(z)\in c\R\}$ consists of at most $n$ points.}
\end{equation}

We use the above property to prove the following lemma.

\begin{lemma}\label{strpo} If the oriented path $\Gamma$ in $\C$ is
an interval of a straight line, $f$ is a polynomial of degree at
most $m\in\Z_+$ and $g:\Gamma\to \C$ is a continuous map taking
values in a line in $\C$ passing through zero such that
$f(z)+g(z)\neq 0$ for every $z\in\Gamma$, then
$w(f+g,\Gamma)<\frac{m+1}{2}$.
\end{lemma}

\begin{proof} Since $\Gamma$ is an interval of a straight line we
can parametrize $\Gamma$ by $\phi:[0,1]\to\C$, $\phi(t)=at+b$ with
$a,b\in\C$, $a\neq 0$. Since $g$ takes values in a line in $\C$
passing through zero, there is $c\in\T$ such that $g(z)\in c^{-1}\R$
for $z\in\Gamma$. Since the function $h(t)=\im\,cf(at+b)$ is a
polynomial with real coefficients of degree at most $m$, it either
vanishes identically on $[0,1]$ or has at most $m$ zeros on $[0,1]$.

If $h\equiv0$, then $f+g:I\to\C$ takes values in the line
$c^{-1}\R$. Hence $w(f+g,\Gamma)=0<\frac{m+1}{2}$. If
$h\not\equiv0$, then the set $C=\{t\in[0,1]:h(t)=0\}$ consists of at
most $m$ points. It is easy to see that the set
$C'=\{z\in\Gamma:(f+g)(z)\in c^{-1}\R\}$ coincides with $\{at+b:t\in
C\}$ and therefore $C'$ consists of at most $m$ points. By
(\ref{cros}), $w(f+g,\Gamma)<\frac{m+1}{2}$.
\end{proof}

Finally, we remind the {\it argument principle}.

\medskip

\noindent {\bf Argument Principle.} \ \it Let $U$ be a bounded open
subset of $\C$, whose boundary is a closed oriented path $\Gamma$,
which encircles $U$ counterclockwise. Let also $f:\overline U\to\C$
be a continuous function such that $f$ is holomorphic on $U$ and
$0\notin f(\Gamma)$. Then $w(f,\Gamma)$ is exactly the number of
zeros of $f$ in $U$ counted with multiplicity.\rm

\medskip

We are ready to prove Lemma~\ref{techn-terr}. Let $k,\delta>0$,
$p\in\pp\setminus\{0\}$ and $m\in\N$. We have to show that for every
sufficiently large $n\in\N$, there exists a connected open set
$W_n\subset \C$ such that $0\in W_n\subseteq\delta\D$ and the
polynomial $q_n=kz((1+z)^n-p)$ has at least $m$ zeros in $W_n$ and
satisfies $q_n(W_n)\subseteq\D$.

Since at most one of the polynomials $q_n$ can be zero, there is
$n_0\in\N$ such that $q_n\neq 0$ for $n\geq n_0$. Let $c>1$ be such
that $|p(z)|\leq c$ for every $z\in\D$. Pick $\alpha\in (0,1)$ such
that $\alpha<\delta$, $\alpha<\frac1{3kc}$, the circle $(\sin
\alpha)\T$ contains no zeros of $p$ and the rays
$\{-1+te^{i\alpha}:t>0\}$ and $\{-1+te^{-i\alpha}:t>0\}$ contain no
zeros of $q_n$ for every $n\geq n_0$. For every $n\in\N$, let
$\epsilon_n=(2c)^{1/n}$. Clearly $\{\epsilon_n\}$ is a strictly
decreasing sequence of positive numbers convergent to $1$. Now for
each $n\in\N$, we consider the open set $W_n\subset\C$ defined by
the formula:
\begin{equation*}
W_n=\{-1+re^{i\beta}:-\alpha<\beta<\alpha,\
\cos\beta-\sqrt{\cos^2\beta-\cos^2\alpha}<r<\epsilon_n\}.
\end{equation*}
It is easy to see that $W_n$ is convex and therefore connected, open
and contains $0$. The following picture shows the set $W_n$.

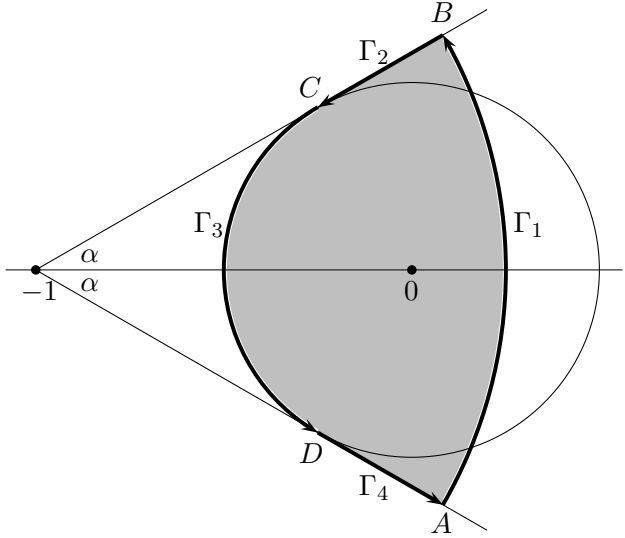
\def\vvv{\vrule width0pt height10pt depth 6pt}
\begin{pspicture}(-.5,-4)(7,4)
\psarc*[linewidth=0pt,linecolor=lightgray](5,0){2.46}{120}{240}
\psarc*[linewidth=0pt,linecolor=lightgray](0,0){6.21}{-30}{30}
\pspolygon*[linewidth=0pt,linecolor=lightgray](3.75,2.165)(5.413,3.125)(5.413,-3.125)(3.75,-2.165)
\psline[linewidth=.3pt](-.4,0)(7.8,0)
\psline[linewidth=.4pt](0,0)(6,3.464)
\psline[linewidth=.4pt](0,0)(6,-3.464)
\psline[linewidth=1.5pt]{<-}(3.75,2.165)(5.413,3.125)
\psline[linewidth=1.5pt]{->}(3.75,-2.165)(5.413,-3.125)
\pscircle[linewidth=.4pt](5,0){2.5}
\psarc[linewidth=1.5pt]{->}(5,0){2.5}{120}{240}
\psarc[linewidth=1.5pt]{->}(0,0){6.25}{-30}{30} \put(-.2,-.4){$-1$}
\psdot(5,0)\psdot(0,0)\put(4.9,-.4){$0$} \put(.6,.1){$\alpha$}
\put(.6,-.28){$\alpha$}  \put(5.25,-3.5){$A$} \put(5.25,3.3){$B$}
\put(3.5,2.3){$C$}\put(3.5,-2.57){$D$} \put(6.35,0.5){$\Gamma_1$}
\put(2.1,.5){$\Gamma_3$} \put(4.3,2.8){$\Gamma_2$}
\put(4.3,-3){$\Gamma_4$} \put(9.2,0){$\begin{array}{l}\llap{{\rm with\ \,}}\text{$W_n$ being the gray area,}\vvv \\ A=-1+\epsilon_ne^{-i\alpha},\vvv\\
B=-1+\epsilon_ne^{i\alpha}\vvv\\ C=-1+(\cos\alpha)e^{i\alpha},\vvv
\\ D=-1+(\cos\alpha)e^{-i\alpha}, \vvv\\ \Gamma_1=\{-1+\epsilon_ne^{it}:t\in[-\alpha,\alpha]\},\vvv
\\ \Gamma_2=\{-1-te^{i\alpha}:t\in[-\epsilon_n,-\cos\alpha]\},\vvv \\
\Gamma_3=\{(\sin\alpha)e^{it}:t\in[\frac\pi2+\alpha,\frac{3\pi}2-\alpha]\}\vvv
\\ \llap{{\rm and\ \,}}\Gamma_4=\{-1+te^{-i\alpha}:t\in[\cos\alpha,\epsilon_n]\}.\vvv
\end{array}$}
\end{pspicture}

The boundary $\partial W_n$, oriented in such a way that it
encircles $W_n$ counterclockwise, is the concatenation of 4 oriented
paths $\partial W_n=\Gamma_1\cup\Gamma_2\cup\Gamma_3\cup\Gamma_4$
defined above. Clearly $\Gamma_1$ is an arc of the circle
$-1+\epsilon_n\T$, $\Gamma_3$ is an arc of the circle
$(\sin\alpha)\T$, while $\Gamma_2$ and $\Gamma_4$ are intervals of
the straight lines $-1+e^{i\alpha}\R$ and $-1+e^{-i\alpha}\R$
respectively. In each case the parametrization is chosen to agree
with the right orientation. First, observe that the farthest from
$0$ points of $\partial W_n$ are $B=-1+\epsilon_ne^{i\alpha}$ and
$A=-1+\epsilon_ne^{-i\alpha}$. Hence $W_n$ is contained in the disk
$|-1+\epsilon_ne^{i\alpha}|\,\D$. Since
$|-1+\epsilon_ne^{i\alpha}|\to
|-1+e^{i\alpha}|=2\sin\frac\alpha2<\alpha$ as $n\to\infty$, we have
\begin{equation}\label{wn1}
\text{$W_n\subset\alpha\D\subset\delta\D$ for each sufficiently
large $n$.}
\end{equation}
Since $\alpha<1$, we also have $W_n\subset\D$ for $n$ large enough.
Since $|p(z)|\leq c$ for $z\in\D$, $|(1+z)^n|\leq 2c$ for $z\in
-1+\epsilon_n\D$ and $W_n\subset -1+\epsilon_n\D$, we see that
$|(1+z)^n-p(z)|\leq 3c$ for all $z\in W_n$ for all sufficiently
large $n$. Since $\alpha<\frac1{3kc}$ and $\sup\limits_{z\in
W_n}|z|<\alpha$ for all $n$ large enough, we have
$|q_n(z)|<k\alpha|(1+z)^n-p(z)|\leq 3ck\alpha<1$ for $z\in W_n$ for
all sufficiently large $n$. Hence
\begin{equation}\label{wn2}
\text{$q_n(W_n)\subseteq\D$ for each sufficiently large $n$.}
\end{equation}
According to (\ref{wn1}) and (\ref{wn2}), it suffices to show that
$r_n=(1+z)^n-p$ has at least $m$ zeros in $W_n$ for each
sufficiently large $n$. Since $r_n$ have no zeros on the rays
$\{-1+te^{i\alpha}:t>0\}$ and $\{-1+te^{-i\alpha}:t>0\}$ for every
$n\geq n_0$, $r_n$ have no zeros on $\Gamma_2\cup\Gamma_4$ for all
$n$ large enough. Since $|(1+z)^n|=2c$ for $z\in \Gamma_1$ and
$|p(z)|\leq c$ for $z\in\Gamma_1$ ($\Gamma_1\subset\D$ for $n$ large
enough), we see that $r_n(z)\neq 0$ for $z\in\Gamma_1$ for all
sufficiently large $n$. Since $\Gamma_3\subset (\sin\alpha)\T$ and
$p$ has no zeros on the circle $(\sin\alpha)\T$,
$\min\limits_{z\in\Gamma_3}|p(z)|=c_0>0$. It is easy to see that
$\Gamma_3$ does not depend on $n$ and is a compact subset of the
disk $-1+\D$. Hence $(1+z)^n$ converges uniformly to 0 on $\Gamma_3$
as $n\to\infty$. Thus $|p(z)|>|(1+z)^n|$ and therefore $r_n(z)\neq
0$ for $z\in\Gamma_3$ for all $n$ large enough. Summarizing, we see
that
\begin{equation*}
\text{$0\notin r_n(\partial W_n)$ for each sufficiently large $n$.}
\end{equation*}
By the argument principle and (\ref{addi}), the number $\nu(n)$ of
zeros of $r_n$ in $W_n$ satisfies
\begin{equation}\label{wn4}
\nu(n)=w(r_n,\partial W_n)=\sum_{j=1}^4 w(r_n,\Gamma_j)\ \ \text{for
all sufficiently large $n$}.
\end{equation}
Since on each of $\Gamma_2$ and $\Gamma_4$, the function $(1+z)^n$
takes values in a line in $\C$ passing through zero and $\Gamma_2$
and $\Gamma_4$ are intervals of straight lines, Lemma~\ref{strpo}
implies that
\begin{equation}\label{wn5}
\textstyle |w(r_n,\Gamma_2)|< \frac{\deg p+1}2\ \ \text{and}\ \
|w(r_n,\Gamma_4)|< \frac{\deg p+1}2\ \ \text{for every sufficiently
large $n$.}
\end{equation}
Since $|(1+z)^n|<|p(z)|$ for $z\in\Gamma_3$ for any $n$ large
enough, (\ref{appro}) implies that
\begin{equation}\label{wn6}
\textstyle |w(r_n,\Gamma_3)|< |w(p,\Gamma_3)|+\frac12\ \ \text{for
every sufficiently large $n$.}
\end{equation}
Finally, since $|p(z)|<|(1+z)^n|$ for $z\in\Gamma_1$ for any $n$
large enough, (\ref{appro}) implies that
$$
\textstyle w(r_n,\Gamma_1)>w((1+z)^n,\Gamma_1)-\frac12\ \ \text{for
every sufficiently large $n$.}
$$
A  direct computation shows that $w((1+z)^n,\Gamma_1)=2n\alpha$.
Hence by the last display,
\begin{equation}\label{wn7}
\textstyle w(r_n,\Gamma_1)>2n\alpha-\frac12\ \ \text{for every
sufficiently large $n$.}
\end{equation}
Combining (\ref{wn4}--\ref{wn7}), we get
$$
\nu(n)>2n\alpha-2-|w(p,\Gamma_3)|-\deg p\ \ \ \text{for every
sufficiently large $n$}.
$$
Since $\Gamma_3$ does not depend on $n$, $\nu(n)\to\infty$ as
$n\to\infty$. Hence $r_n$ and therefore $q_n$ has at least $m$ zeros
in $W_n$ for each $n$ large enough. The proof of
Lemma~\ref{techn-terr} and that of Theorem~\ref{main} is complete.

\section{Remarks and open questions}

\noindent{\bf 1.}\ \ Our construction of a chaotic Banach algebra
provides little control over its Banach space structure. Thus the
following interesting questions arise.

\begin{question}\label{q1} Which separable infinite dimensional
Banach spaces admit a multiplication turning them into a supercyclic
or into an almost hypercyclic Banach algebra? In particular, is
there a multiplication on $\ell_2$, turning it into a chaotic Banach
algebra?
\end{question}

\noindent{\bf 2.}\ \ The structural properties of the class of
supercylic or almost hypercyclic Banach algebras remain a complete
mystery.

\medskip

\noindent{\bf 3.}\ \ Let $\hh$ be the Hilbert space of
Hilbert--Schmidt operators on $\ell_2$. With respect to the
composition multiplication, $\hh$ is a non-commutative non-unital
Banach algebra. Let also $S\in\hh$ be defined by its action on the
basic vectors as follows: $Se_0=0$, $Se_n=n^{-1}e_{n-1}$ if $n\geq
1$. Consider the left multiplication by $S$ operator $\Phi\in
L(\hh)$, $\Phi(T)=ST$. Using the hypercyclicity and supercyclicity
criteria \cite{bama-book}, it is easy to see that $\Phi$ is
supercyclic and $I+\Phi$ is hypercyclic. Thus supercyclicity of a
multiplication operator and hypercyclicity of a perturbation of the
identity by a multiplication operator on a non-commutative Banach
algebra is a much simpler phenomenon.

\medskip

\noindent{\bf 4.}\ \ We would also like to raise the following
question. We say that a Banach algebra $A$ is {\it wildly chaotic}
if it has a supercyclic element $a$ such that for every $z\in\T$,
the set $\{a(z+a)^n:n\in\N\}$ is dense in $A$.

\begin{question}\label{q2} Does there exist a wildly chaotic infinite
dimensional Banach algebra?
\end{question}

\noindent Note that our construction can be modified to make
$\{a(z+a)^n:n\in\N\}$ dense in $A$ for each $z$ from a given
countable subset of $\T$.

\medskip

\noindent{\bf 5.}\ \ Corollary~\ref{smallsp} ensures the existence
of a hypercyclic operator $T$ with $\sigma(T)=\{1\}$  and $T\oplus
T$ being non-cyclic. This naturally leads to the question whether
such operators exist on every separable infinite dimensional Banach
space.

\begin{question}\label{q3} Let $X$  be a separable infinite dimensional Banach
space. Does there exist a $T\in L(X)$ such that $T$ is hypercyclic,
$T\oplus T$ is non-cyclic and $\sigma(T)=\{1\}$? What is the answer
for $X=\ell_2$?
\end{question}

The above question is related to the following question of Bayart
and Matheron \cite{bama}.

\begin{question}\label{q4} Does every separable infinite dimensional Banach
space admit a hypercyclic operator $T$ such that $T\oplus T$ is
non-cyclic?
\end{question}

\noindent{\bf 6.}\ \ Bayart and Matheron \cite{bama-book} ask
whether there exists a hypercyclic strongly continuous operator
semigroup $\{T_t\}_{t\geq 0}$ on a Banach space $X$ such that the
semigroup $\{T_t\oplus T_t\}_{t\geq 0}$ acting on $X\oplus X$ is
non-hypercyclic. As we have already mentioned, Theorem~\ref{main}
provides a quasinilpotent operator $M_a$ on the Banach space $A$
such that $I+M_a$ is hypercyclic, while $(I+M_a)\oplus (I+M_a)$ is
non-hypercyclic. Since $M_a$ is quasinilpotent,
$$
S=\ln(I+M_a)=\sum_{n=1}^\infty \frac{(-1)^{n-1}}{n}M_a^n
$$
is a well-defined (also quasinilpotent) continuous linear operator
on $A$. Hence we can consider the operator norm continuous semigroup
$\{e^{tS}\}_{t\geq 0}$, which contains all powers of $I+M_a$:
$e^{nS}=(I+M_a)^n$ for $n\in\N$. It follows that $\{e^{tS}\}_{t\geq
0}$ is hypercyclic. On the other hand, $e^S\oplus e^S=(I+M_a)\oplus
(I+M_a)$ is a non-hypercyclic member of the semigroup
$\{e^{tS}\oplus e^{tS}\}_{t\geq 0}$. According to Conejero, M\"uller
and Peris \cite{semi}, $T_t$ is hypercyclic for every $t>0$ if
$\{T_t\}_{t\geq 0}$ is a hypercyclic strongly continuous operator
semigroup. Hence $\{e^{tS}\oplus e^{tS}\}_{t\geq 0}$ is
non-hypercyclic which answers negatively the above mentioned
question of Bayart and Matheron.

\small\rm








\end{document}